\documentclass{article}
\usepackage{amsthm, amsfonts, amssymb,latexsym}

\newcommand{\F}{\mathbb{F}}

\renewcommand{\u}{\underline}
\usepackage{amsmath}
\usepackage[a4paper, total={6in, 8in}]{geometry}
\newcommand{\cN}{\mathcal{N}}
\newcommand{\cK}{\mathcal{K}}

\newcommand{\cH}{\mathcal{H}}
\newcommand{\cE}{\mathcal{E}}
\newcommand{\cC}{\mathcal{C}}
\newcommand{\cP}{\mathcal{P}}
\newcommand{\cS}{\mathcal{S}}

\newcommand{\M}{\text{Mult}}
\newtheorem{proposition}{Proposition}
\newtheorem{lemma}{Lemma}

\newtheorem{theorem}{Theorem}
\usepackage{color}
\newcommand{\tbd}[1]{{{\color{black}#1}}}

\begin{document}
\title{Conical Kakeya and Nikodym Sets in Finite Fields}
\author{Audie Warren and Arne Winterhof}
\date{}

\maketitle

\begin{center}
Johann Radon Institute for Computational and Applied Mathematics,
Austrian Academy of Sciences, Altenberger Str. 69, 4040 Linz, Austria,
E-mail: \{audie.warren,arne.winterhof\}@oeaw.ac.at
\end{center}

\begin{abstract}
 A Kakeya set contains a line in each direction. Dvir proved a lower bound on the size of any Kakeya set in a finite field using the polynomial method.
 We prove analogues of Dvir's result for non-degenerate conics, that is, parabolae \tbd{and} hyperbolae \tbd{ (but not ellipses which do not have a direction)}.
 We also study so-called \tbd{conical }Nikodym sets where a small variation of the proof provides a lower bound on their sizes. \tbd{(Here ellipses are included.)}
 
 \tbd{Note that the bound on conical Kakeya sets has been known before, however, without an explicitly given constant which is included in our result and close to being best possible.}
\end{abstract}


\section{Introduction}
\label{intro}

A subset $\cK\subseteq\F_q^n$ of $n$-dimensional vectors over the finite field $\F_q$ of $q$ elements is called a {\em Kakeya set} in $\F_q^n$ if it contains a line in each direction. Using the polynomial method Dvir \cite[Theorem~1.5]{dv} showed that any Kakeya set in $\F_q^n$ contains
at least $c_nq^n$ elements with a constant $c_n$ depending only on $n$, see also \cite[Theorem~2.11]{gu}.

A set $\cN\subseteq\F_q^n$ is called a {\em Nikodym set} in $\F_q^n$ if for each point $x\in \F_q^n$ there is a line $L$ containing~$x$ such that $L\setminus\{x\}\subseteq \cN$. A small variation of Dvir's proof also provides that any Nikodym set in $\F_q^n$ contains
at least $c_nq^n$ elements with a constant $c_n$ depending only on $n$, see \cite[Theorem~2.9]{gu}.

 Now let $q$ be the power of an odd prime. The set of zeros $(x,y)\in \F_q^2$ of a polynomial 
 $$Q(X,Y)=AX^2+BXY+CY^2+DX+EY+F,\quad A,B,C,D,E,F\in \F_q,$$
 of degree $2$, that is $A$, $B$ and $C$ are not all zero,
 is called a {\em conic} in $\F_q^2$.
In the degenerate case, that is $Q(X,Y)$ is reducible over the algebraic closure of $\F_q$, we get a {\em pair of} (intersecting, parallel or identical) {\em lines}, a {\em point} or the {\em empty set}. We restrict ourselves to the non-degenerate case, that is, $Q(X,Y)$ is absolutely irreducible over~$\F_q$,
since the degenerate case is either trivial or can be reduced to the previously studied case of a single line.  
We may assume $C=1$, $C=0$ and $A=1$, or $A=C=0$ and $B=1$. After regular affine substitutions 
$$\left(\begin{array}{c} X\\ Y\end{array}\right)\mapsto \left(\begin{array}{c} aX+bY+e\\ cX+dY+f\end{array}\right), \quad ad\ne bc,$$ 
we are left with the following cases where $g$ is any fixed non-square in $\F_q^*$:
\begin{itemize} 
\item $A=C=0$, $B=1$:  {\em hyperbola} $\{(x,y)\in \F_q^2 : x\ne 0 \mbox{ and }y=x^{-1}\}=\{(t,t^{-1})\in \F_q^2 : t\in \F_q^*\}$.
\item $(A,C)\in \{(1,0),(0,1)\}$: {\em parabola} $\{(x,y)\in \F_q^2 : y=mx^2\}=\{(t,mt^2): t\in \F_q\}$, where $m\in\{1,g\}$.
\item $C=1$, $A\ne 0$: {\em ellipse} $\{(x,y)\in \F_q^2 : y^2=\tbd{g}x^2+k\}$, $k\in \F_q^*$.\\
\tbd{(Note that conics defined by $Y^2=X^2+k$ can be transformed into the form $XY=1$ and are hyperbolae.)}
\end{itemize}

For parabolae and hyperbolae the parametrisations 
$$(x(t),y(t))=(t,mt^2), \quad t\in \F_q,$$ 
and 
$$(x(t),y(t))=(t,t^{-1}),\quad  t\in \F_q^*,$$ respectively, are obvious. However, we
can also derive parametrisations of ellipses $(x(t),y(t))$ where $t\in\F_{q^2}$ with $t^{q+1}=1$, see Section~\ref{ell} below.

To extend the definition of a conic to a general dimension $n \ge 2$, we 
embed any conic in $\F_q^2$ into a plane in $\F_q^n$. That is
for some vectors \tbd{$\underline{a}, \underline{b},\underline{c} \in \F_q^n$} where 
\tbd{$\u{b}$ and $\u{c}$ are linearly independent}:
 \begin{itemize}
      \item A(n embedding of a) {\em hyperbola} in $\F_q^n$ is a set 
      \begin{equation}\label{H} \cH=\{\underline{a} + t\underline{b} + t^{-1}\underline{c} : t\in \F_q^*\}.\end{equation}
      \item A(n embedding of a) {\em parabola} in $\F_q^n$ is a set \begin{equation}\label{P} \cP=\{\underline{a} + t\underline{b} + t^2\underline{c} : t\in \F_q\}.\end{equation}
      \item An (embedding of an) {\em ellipse} in $\F_q^n$ is a set \begin{equation}\label{E} \cE=\{\underline{a} + x(t)\underline{b} +y(t) \underline{c}: \tbd{t \in \F_{q^2}, t^{q+1}=1}\}\end{equation} where 
       $(x(t),y(t))\tbd{\in \F_q^2}$ is given in Section~\ref{ell}.
 \end{itemize}
\tbd{(Without the linear independence of $\u{b}$ and $\u{c}$ the embedding can have fewer points than the embedded conic. Hence, a hyperbola has $q-1$ points, a parabola $q$ points and an ellipse $q+1$ points.)}
 
We give adaptations of Dvir's proof to give bounds on conical Kakeya and Nikodym sets defined as follows. 

A subset $\cN \subseteq \F_q^n$ is called a {\em conical
 Nikodym set} if for all $\underline{x} \in \F_q^n$ there is a non-degenerate conic $\cC$ of the form $(\ref{H})$, $(\ref{P})$ or $(\ref{E})$
 with $\underline{x} \in \cC$ 
and 
 $\cC\setminus \{\underline{x}\} \subseteq \cN$.
 
 In order to define conical Kakeya sets, we must decide on how to define the 'direction' of a conic
 \tbd{which can be identified with the 'point(s) at infinity' of the conic, that is, a hyperbola has two directions $\u{b}$ and $\u{c}$,
 a parabola has one direction $\u{c}$, and an ellipse has no direction. 
 }

 A subset $\cK \subseteq \F_q^n$ is called a 
{\em conical Kakeya set}
if for all $\underline{d} \in \F_q^n \setminus \{\underline{0}\}$ \\
 there exist $\underline{a}, \underline{b}, \tbd{\u{c}} \in \F_q^n$ such that $\underline{b}$ and \tbd{$\underline{c}$} are linearly independent and there is a conic contained in~$\cK$
either of the form $(\ref{H})$ \tbd{with $\underline{d}\in \{\underline{b},\underline{c}\}$} or of the form $(\ref{P})$
with \tbd{$\underline{d} = \underline{c}$}.

We prove the following Theorem.

 \begin{theorem}\label{main}
 Let $\cS \subseteq \F_q^n$ with $n\ge 2$ be a conical Kakeya or Nikodym set, where $q$ is a power of an odd prime. Then
 $$|\cS| \geq \left( \frac{q-1}{2n} \right)^n.$$
 \end{theorem}

\tbd{For conical Kakeya sets the lower bound $c_nq^n$ with a constant depending on $n$ follows from 
\cite[Corollary~1.10]{elobta}. However, in contrast to \cite{elobta} our constant is explicit and in 
Section~\ref{multi} we use the {\it method of multiplicities} of \cite{dvkosasu} to determine a constant of the
form $c_n=c^n$ where $c$ does not depend on $n$.}

\tbd{Moreover, at the end of the paper we give an example of a subset of $\F_q^2$ of size $q+1$ which contains for each~$\u{c}$, resp., $\u{b}$ an ellipse of form $(\ref{E})$. 
Hence, it is necessary to exclude ellipses in the definition of conical Kakeya sets.}

In Section~\ref{ell} we derive a parametrisation for ellipses needed in the proof of Theorem~\ref{main}. In Section~\ref{niko} we prove Theorem~\ref{main} for conical Nikodym sets and in Section~\ref{kake} for conical Kakeya sets.  \tbd{In Section~\ref{multi} we improve the constant $c_n$ using the method of multiplicities.}
In Section~\ref{remarks} we conclude with some final remarks.

\tbd{For readers not familiar with the polynomial method we refer to the book of Guth \cite{gu} and the survey article of Tao \cite{ta} as excellent starting points.} 

 \section{Parametrisation of ellipses}
 \label{ell}
 
In this section we derive a parametrisation for ellipses, which is vital in our proof of Theorem \ref{main}
\tbd{for elliptic Nikodym sets}.

  Consider the ellipse $\cE = \{ (x,y) \in \F_q^2 : y^2 = \tbd{g}x^2 + k \}$, where \tbd{$g$} is a non-square in $\F_q^*$ and $k\in \F_q^*$. 
  By \cite[Lemma 6.24]{lini} we have 
  \begin{equation}\label{nop} |\cE|=\tbd{q+1}.
  \end{equation}
  Using analogues $s(t)$ and $r(t)$ of sine and cosine for finite fields defined below, see for example \cite[Definition~15.5]{pa}, we are able to find parametrisations of ellipses.

 \tbd{Note that a solution $z$ of $z^2=g$} is not an element of $\F_q$: $z\in \F_{q^2}\setminus\F_q$.
 \tbd{Let $(u,v)\in \F_q^2$ be any fixed solution of} 
 $$v^2=\tbd{g}u^2+k$$
 \tbd{which exists by $(\ref{nop})$.} \tbd{Then verify that}
 \begin{eqnarray*} s(t)&=&2^{-1}z(t-t^{-1}),\\
r(t)&=&2^{-1}(t+t^{-1}),
\end{eqnarray*}
\tbd{is a solution of $s(t)^2=g(r(t)^2-1)$.}
 It can be \tbd{easily checked} that
 $$\cE=\{(x(t),y(t)) : t\in \F_{q^2}^*, t^{q+1}=1\}$$
 with 
\begin{eqnarray*}
  x(t) &=& \tbd{g}^{-1}vs(t)+ur(t),\\
  y(t) &=& us(t)+vr(t).
 \end{eqnarray*}
 Since $r(t)^q=r(t)$ and $s(t)^q=s(t)$ (using $z^q=z\tbd{g}^{(q-1)/2}=-z$ because $\tbd{g}$ is a non-square in $\F_q^*$) we have $(r(t),s(t))\in \F_q^2$, so that $(x(t),y(t)) \in \F_q^2$.

 \section{Conical Nikodym sets}
 \label{niko}

 In this section we prove Theorem~\ref{main} for conical Nikodym sets. 
    
  \begin{proposition} 
  Let $\cN \subseteq \F_q^n$ with $n\ge 2$ and $q$ the power of an odd prime be a conical Nikodym set. Then we have
  $$|\cN| \geq \left( \frac{q-1}{2n} \right)^n.$$
  \end{proposition}

  \begin{proof} 
 Suppose $|\cN| < \left(\frac{q-1}{2n}\right)^n$. By \cite[Lemma~2.4]{gu}, there is a non-zero polynomial $f$ with $f(\tbd{\u{s}}) = 0$ for all $\tbd{\u{s}} \in \cN$, and $\deg(f) \leq n|\cN|^{1/n}\le \frac{q-3}{2}$. Take any $\underline{x} \in \F_q^n$. As $\cN$ is conical Nikodym, there exists a conic $\cC$ of the form $(\ref{H})$, $(\ref{P})$ or $(\ref{E})$ with $\underline{x} \in \cC$ and $\cC \setminus \{ \underline{x} \} \subseteq \cN$. We split into cases depending on the form of the conic $\cC$.
  
  Firstly assume the conic $\cC$ is a parabola $\cP$. Parametrise this parabola as $$\cP = \{\tbd{\underline{a}} + t \underline{b} + t^2 \underline{c} : t \in \F_q \}.$$
 Applying these points to the polynomial $f$, we define $F(t) = f(\underline{a} + t \underline{b} + t^2 \underline{c})$ a univariate polynomial in $t$ of degree $\deg(F) \leq q-3$. We also know that it has $q-1$ zeros corresponding to the points of the parabola lying in $\cN$, and thus must be zero on the whole parabola, in particular $f(\underline{x}) = 0$.
 
 Secondly assume $\cC$ is a hyperbola $\cH$. Parametrise this hyperbola as $$\cH = \{\underline{a} + t \underline{b} + t^{-1} \underline{c}: t \in \F_q^* \}.$$
 Applying these points to the polynomial $f$, we define $F(t) = t^{\deg(f)}f(\underline{a} + t \underline{b} + t^{-1} \underline{c})$ a univariate polynomial in $t$ of degree $\deg(F) \leq q-3$. We also know that it has $q-2$ zeros corresponding to the points of the hyperbola $\cH$ lying in $\cN$, and thus must be zero on the whole hyperbola. Again we find $f(\underline{x}) = 0$.

 Thirdly we assume that $\cC$ is an ellipse $\cE$. The number of points on this ellipse \tbd{is $q+1$}. We use our parametrisation of an ellipse; it has form 
 $$\cE = \{ \underline{a} +  \underline{b}x(t) +\underline{c}y(t) : \tbd{t\in \F_{q^2}^*,t^{q+1}=1}\}$$
 for some appropriate choice of $\underline{a}$, $\underline{b}$ and $\underline{c}$, and $(x(t),y(t))$ are given in Section \ref{ell}.  We consider the polynomial $F(t) = t^{\deg(f)}f(\underline{a} + \underline{b}x(t) +  \underline{c}y(t))$. \tbd{This} polynomial is univariate in $t$ of degree $\deg(F) \leq q-3$. 
 We know that it has \tbd{$q$} zeros \tbd{(in $\F_{q^2}$)} corresponding to the points of the ellipse $\cE$ lying in $\cN$, and thus must be zero on the whole ellipse. We again find that $f(\underline{x}) = 0$.
 
 In all three cases we found that $f(\underline{x}) = 0$. As $\underline{x}$ was chosen arbitrarily we conclude that $f(\underline{x})=0$ for all $\underline{x}\in \F_q^n$. As $\deg(f) \leq \frac{q-3}{2}$ the polynomial $f$ must be the zero polynomial, a contradiction.
   \end{proof}

\section{Conical Kakeya sets}
\label{kake}

In this section we prove Theorem~\ref{main} for conical Kakeya sets.

\begin{proposition}
Let $\cK \subseteq \F_q^n$ with $n\ge 2$ and $q$ the power of an odd prime be a conical Kakeya set. Then
$$|\cK| \geq \left(\frac{q-1}{2n}\right)^n.$$
\end{proposition}

\begin{proof}
Suppose that $|\cK| < \left( \frac{q-1}{2n} \right)^n$. By \cite[Lemma~2.4]{gu} there exists $f$ a non-zero polynomial with $f(\tbd{\u{s}})=0 \ $ for all $\tbd{\u{s}} \in \cK$, with degree $d  \leq \frac{q-3}{2}$. We split this polynomial into a sum of its greatest degree part and the lower degree terms as 
 $$f = f_d + g, \quad \deg(f_d) = d, \quad \deg(g) < d, \quad f_d(x_1,\ldots,x_n) = \sum_{i_1 + \ldots+i_n = d}e_{i_1,\ldots,i_n}x_1^{i_1}x_2^{i_2}\cdots x_n^{i_n}.$$
Note that as $f_d$ is homogeneous, $f_d(\underline{0})=0$. Take any $\underline{x} \in \F_q^n \setminus \{ \underline{0} \}$. As $\cK$ is conical Kakeya, there exists some conic $\cC$ of the form $(\ref{H})$ 
\tbd{or} $(\ref{P})$ with $\underline{x}$ appearing as $\underline{c}$ \tbd{for parabolae} 
\tbd{and} $\underline{b}$ or $\underline{c}$ for hyperbolae from Section \ref{intro}. We split into \tbd{two} cases depending on which type of conic $\cC$ defines. 

First assume $\cC$ is a parabola $\cP$. It has parametrisation
$$\cP = \{ \underline{a} + t\underline{b} + t^2\underline{c}:t \in \F_q \}. $$
We consider the polynomial $F(t) \tbd{=} f(\underline{a} + t\underline{b} + t^2\underline{c})$, which is univariate in $t$ of degree $2d$. Since $f$ is zero on $\cK$, $F(t) = 0$ for all $t\in \F_q$. Then as $\deg(F) = 2d < q$, $F$ is identically zero. We note that the coefficient of $t^{2d}$ in $F(t)$ is the coefficient of $t^{2d}$ in $f_d(\underline{a} + \underline{b}t + \underline{c}t^2)$:
 \begin{align*}
     f_d(\underline{a} + \underline{b}t + \underline{c}t^2) &= f_d(a_1+b_1t + c_1t^2,\ldots,a_n+b_nt + c_nt^2) \\
     & = \sum_{i_1 + \ldots+i_n = d}e_{i_1,\ldots,i_n}(a_1+b_1t + c_1t^2)^{i_1}(a_2+b_2t  +c_2t^2)^{i_2}\cdots (a_n+b_nt + c_nt^2)^{i_n}.
 \end{align*}
 Upon multiplying out to find the coefficient of $t^{2d}$ we have
 \begin{align*}
     \sum_{i_1 + \ldots+i_n = d}e_{i_1,\ldots,i_n}c_1^{i_1}c_2^{i_2}\cdots c_n^{i_n}t^{2d} & = t^{2d}  \sum_{i_1 + \ldots+i_n = d}e_{i_1,\ldots,i_n}c_1^{i_1}c_2^{i_2}\cdots c_n^{i_n}
     \\ & = t^{2d} f_d(\underline{c})
 \end{align*}
and thus as $F$ is identically zero, $f_d(\underline{c}) = f_d(\underline{x}) =0$.

Secondly we assume $\cC$ is a hyperbola $\cH$. \tbd{Up to the relabelling of $t \rightarrow t^{-1}$, we may assume it has parametrisation} 
$$\cH=\{\underline{a} + t\underline{b} + t^{-1}\underline{c} : t\in \F_q^*\}.$$
Consider the univariate polynomial $F(t) = t^df(\underline{a} + t \underline{b}+t^{-1} \underline{c})$, which is of degree $\deg(F) = 2d < q-1$. 
The polynomial $f$ vanishes on $\cK$, and so $F(t) =0$ for all $t \in \F_q^*$. As $\deg(F) < q-1$ with $F$ having at least $q-1$ zeros, we have that $F(t)$ is identically zero, in particular its constant term is zero. We calculate the constant term of $F(t)$; it is precisely the coefficient of $t^{-d}$ in $f_d(\underline{a}+ \underline{b}t + \underline{c}t^{-1} )$,
\begin{align*}
    f_d(\underline{a} + \underline{b}t + \underline{c}t^{-1} ) & = \sum_{i_1 +\ldots+i_n = d}e_{i_1,\ldots,i_n}(a_1+b_1t + c_1t^{-1})^{i_1}\cdots (a_n+b_nt + c_nt^{-1})^{i_n}
\end{align*}
so the coefficient of $t^{-d}$ is $$ \sum_{i_1 +\ldots+i_n = d}e_{i_1,\ldots,i_n}c_1^{i_1}\cdots c_n^{i_n} = f_d(\underline{c}).$$ Thus $f_d(\underline{c})= f_d(\underline{x})=0$.

In \tbd{both} cases we have $f_d(\underline{x})=0$. Since we already knew that $f_d(\underline{0})=0$, we have $f_d(\underline{x}) = 0$ for all $\underline{x} \in \F_q^n$. As $d < q$, $f_d$ is identically zero, which is a contradiction.
\end{proof}

\tbd{\section{Improvements via the method of multiplicities}\label{multi}

The 'method of multiplicities' was used in \cite{dvkosasu}, see also \cite{ta}, to prove a constant of $2^{-n}$ for line Kakeya sets. This involves Hasse derivatives and exploiting polynomials which vanish to a high multiplicity on a particular set. 

Let $\u{x} = (x_1,...,x_n)$ and $f \in \F_q[\underline{x}]$. For a vector $\underline{i}=(i_1,\ldots,i_n) \in \mathbb{N}^n$, the $\underline{i}$'th {\it Hasse derivative} of $f$, which we denote $f^{\u{i}}(\u{x})$, is the coefficient of $\u{y}^{\u{i}}$ in the polynomial $f(\u{x} + \u{y})$, where $\u{y}^{\u{i}}$ is the monomial $y_1^{i_1}y_2^{i_2}...y_n^{i_n}$.

For $f \in \F_q[\u{x}]$ and $\u{a} \in \F_q^n$, the {\it multiplicity} of $f$ at $\u{a}$, denoted $\M(f,\u{a})$, is the largest integer $M$ such that for all vectors $\u{i} \in \mathbb{N}^n$ of weight $wt(\u{i}) < M$, the $\u{i}$'th Hasse derivative of $f$ is zero at $\u{a}$, that is, $f^{\u{i}}(\u{a}) = 0$,
where $wt(\u{i})=i_1+\ldots+i_n$.

We make use of five results relating to multiplicities and Hasse derivatives. These results, with proofs, can be found in \cite{dvkosasu}, see also \cite{ta}. 

\begin{lemma} \label{commute}
Hasse derivatives 'commute' with taking homogeneous parts of highest degree. That is, for $f \in \F_q[\u{x}]$ of total degree $d$, letting $f_d$ denote the homogeneous part of $f$ of degree $d$, we have
$$(f_d)^{\u{i}}(\u{x}) = (f^{\u{i}})_{d'}(\u{x})$$
where $d'\le d-wt(\u{i})$ is the degree of $f^{\u{i}}$.
\end{lemma}

\begin{lemma} \label{weight}
Taking $\u{i}$'th Hasse derivatives reduces multiplicity by at most the weight of $\u{i}$. That is,
$$\M(f^{\u{i}},\u{a}) \geq \M(f,\u{a}) - wt(\u{i}).$$
\end{lemma}

\begin{lemma} \label{comp}
Multiplicities of compositions of polynomials $f(g(\u{x}))$ at $\u{a}$ is at least the multiplicity of $f$ at $g(\u{a})$. That is,
$$\M(f(g(\u{x})), \u{a}) \geq \M(f(\u{x}), g(\u{a})).$$
\end{lemma}

\begin{lemma}[Vanishing lemma for multiplicities] \label{vanish}
Let $f \in \F_q[\u{x}]$ be of degree $d$. Then $$\sum_{a \in \F_q^n} \M(f, \u{a}) > dq^{n-1} \implies f \ \text{is the zero polynomial.}$$
\end{lemma}

 \begin{lemma} \label{poly}
 Suppose $\cS \subseteq \F_q^n$ such that for some natural numbers $m,d$ we have
 $$|\cS|{m+n-1 \choose n} < {d + n \choose n}.$$
 Then there is a non-zero polynomial $f \in \F_q[\u{x}]$ of degree at most $d$, such that $\M(f,\u{s}) \geq m$ for all~$\u{s} \in \cS$.
 \end{lemma}
Note that Lemma~\ref{poly} is satisfied if $|\cS|\le\left(\frac{d}{m+n}\right)^n$.
 
}

\subsection{Conical Nikodym sets}

In this section we use the method of multiplicities to prove the following theorem.
\begin{theorem}
Let $\cN\subset \F_q^n$ be a conical Nikodym set, with $q$ a power of an odd prime. Then we have
$$|\cN|\ge 
\left(3+\frac{4}{q-2}\right)^{-n}q^n=(3+o(1))^{-n}q^n,\quad q\to \infty.$$
\end{theorem}

 \begin{proof}
We begin by taking a large multiple of $q$, call it $lq$ for some positive integer $l$, and define 
$$m=\left\lfloor \left(3+\frac{4}{q-2}\right)l\right\rfloor.$$
 
 Assume that 
 $|\cN| \le \left(\frac{lq-1 }{m + n}\right)^n$. By Lemma \ref{poly}, there is a non-zero polynomial $f \in \F_q[\u{x}]$ of degree $d < lq$, such that $\M(f, \u{x}) \geq m$ for all $\u{x} \in \cN$. Let $f_d(\u{x})$ be the homogeneous part of $f$ of degree $d$, which we know is not the zero polynomial. We aim to show that $f_d$ has high multiplicity everywhere in $\F_q^n$, and thus must be the zero polynomial. Indeed, we will show it has multiplicity $l$ everywhere.
 
 Choose $\u{i} \in \mathbb{N}^n$ with $wt(\u{i}) < l$, and $\u{z} \in \F_q^n$. We aim to show that $(f_d)^{\u{i}}(\u{z}) = 0$. The case $\u{z} = \u{0}$ is trivial, so we assume  $\u{z} \not= \u{0}$. As $\cN$ is conical Nikodym, there is a conic $\cC$ such that $\u{z} \in \cC$ and $\cP \setminus \{ \u{z} \} \subset \cN$. We split into cases depending on the conic $\cC$, aiming to show that $f^{\u{i}}(\u{z}) = 0$.
 
 \underline{\emph{Case 1 - Parabola}}
 
 Assume $\cC$ is a parabola, which we can parametrise as $\u{c}t^2 + \u{b}t + \u{a}$. We know by the properties of $f$ that $\M(f,\u{c}t^2 + \u{b}t + \u{a}) \geq m$ for $q-1$ values of $t$. By Lemma \ref{weight}, we have $\M(f^{\u{i}},\u{c}t^2 + \u{b}t + \u{a}) \geq m - wt(\u{i})$. We can now use Lemma \ref{comp} to get 
$$\M(f^{\u{i}}(\u{c}x^2 + \u{b}x + \u{a}), t) \geq \M(f^{\u{i}},\u{c}t^2 + \u{b}t + \u{a}) \geq m - wt(\u{i})$$
for $q-1$ values of $t$. Note that the polynomial $f^{\u{i}}(\u{c}x^2 + \u{b}x + \u{a})$ has degree $d' \leq 2\deg(f^{\u{i}}) \leq 2(d - wt(\u{i}))$. However, $f^{\u{i}}(\u{c}x^2 + \u{b}x + \u{a})$ has multiplicity at least $m - wt(\u{i})$ for $q-1$ values of $t$, so that by Lemma \ref{vanish}, as $d < lq$, $wt(\u{i}) < l$, we have 
\begin{align*}
    \sum_{t \in \F_q} \M(f^{\u{i}}(\u{c}x^2 + \u{b}x + \u{a}), t)  & \geq (q-1)(m - wt(\u{i})) \\
    & > 2(d - wt(\u{i})) \\
    & \geq \deg(f^{\u{i}}(\u{c}x^2 + \u{b}x + \u{a}))
\end{align*}
so that $f^{\u{i}}(\u{c}x^2 + \u{b}x + \u{a})$ is in fact the zero polynomial. But then $f^{\u{i}}(\u{z}) = 0$, as needed.

 \underline{\emph{Case 2 - Hyperbola}}
 
 Assume $\cC$ is a hyperbola, which we can parametrise as $\u{b}t + \u{c}t^{-1} + \u{a}$. We know by the properties of $f$ that $\M(f,\u{b}t + \u{c}t^{-1} + \u{a}) \geq m$ for $q-2$ values of $t$. By Lemma \ref{weight}, we have $\M(f^{\u{i}},\u{b}t + \u{c}t^{-1} + \u{a}) \geq m - wt(\u{i})$. We have $ d'=\deg(f^{\u{i}})  \leq d - wt(\u{i})$, and we define the polynomial $F(t) = t^{d'}f^{\u{i}}(\u{b}t + \u{c}t^{-1} + \u{a})$ which has degree $2d'$, and also has multiplicity at least $m- wt(\u{i})$ for $q-2$ values of $t$. By the vanishing lemma, we have
 $$\sum_{t \in \F_q} \M(F,t) \geq (q-2)(m - wt(\u{i})) > 2(d - wt(\u{i}))$$
so that $F(t)$ is the zero polynomial. In particular, when we input the value $t_0 \neq 0$ corresponding to $\u{z}$ on the hyperbola, we get zero. Then $F(t_0) = t_0^{d'}f^{\u{i}}(\u{b}t_0 + \u{c}t_0^{-1} + \u{a}) = t_0^{d'}f^{\u{i}}(\u{z}) = 0 \implies f^{\u{i}}(\u{z}) = 0  $ as needed.

 \underline{\emph{Case 3 - Ellipse}}

Assume $\cC$ is an ellipse, which we can parametrise as $\u{b}x(t) + \u{c}y(t) + \u{a}$ with $\u{b}$ and $\u{c}$ linearly independent. We know by the properties of $f$ that $\M(f,\u{b}x(t) + \u{c}y(t) + \u{a}) \geq m$ for $q$ values of $t$. By Lemma \ref{weight}, we have $\M(f^{\u{i}},\u{b}x(t) + \u{c}y(t) + \u{a}) \geq m - wt(\u{i})$. We have $d'= \deg(f^{\u{i}})  \leq d - wt(\u{i})$, and we define the polynomial $F(t) = t^{d'}f^{\u{i}}(\u{b}x(t) + \u{c}y(t) + \u{a})$ which has degree $2d'$, and also has multiplicity at least $m- wt(\u{i})$ for $q$ values of $t \in \F_{q^2}$. By the vanishing lemma, we have
 $$\sum_{t \in \F_{q^2}} \M(F,t) \geq q(m - wt(\u{i})) > 2(d - wt(\u{i}))$$
 so that $F(t)$ is the zero polynomial. In particular, when we input the value $t_0 \neq 0$ corresponding to $\u{z}$ on the ellipse, we get zero. Then $F(t_0) = t_0^{d'}f^{\u{i}}(\u{b}x(t) + \u{c}y(t) + \u{a}) = t_0^{d'}f^{\u{i}}(\u{z}) = 0 \implies f^{\u{i}}(\u{z}) = 0  $ as needed.

This was for arbitrary $\u{z}$, so we have $\M(f,\u{z}) \geq l$ for all $\u{z} \in \F_q^n$, and we  may use the vanishing lemma a final time to show
$$\sum_{\u{x} \in \F_q^n}\M(f, \u{x}) \geq lq^n > dq^{n-1}$$
so $f$ is in fact the zero polynomial, a contradiction. We may allow $l$ to go to infinity, so that 
$$|\cN| \geq \lim_{l \rightarrow \infty}\left(\frac{lq-1   }{m + n}\right)^n \geq \lim_{l \rightarrow \infty} \left( \frac{q-1/l }{3+4/(q-2)+ n/l} \right)^n= \left(\frac{q}{3+4/(q-2)}\right)^n $$
as needed.
 \end{proof}

\tbd{
\subsection{Conical Kakeya sets}

In this section we adapt the proof of \cite{dvkosasu} for line Kakeya sets to conical Kakeya sets.

\begin{theorem}\label{parabolicKakeya}
Let $\cK\subset \F_q^n$ be a conical Kakeya set, with odd q. Then we have
$$|\cK|\ge \left(\frac{q}{3}\right)^n.$$
\end{theorem}

 \begin{proof}
 We begin by taking a large multiple of $q$, call it $lq$, and define $m = 3l$.
 
 Assume that $|\cK| \le \left(\frac{lq -1  }{m + n}\right)^n$. By Lemma \ref{poly}, there is a non-zero polynomial $f \in \F_q[\u{x}]$ of degree $d < lq$, such that $\M(f, \u{k}) \geq m$ for all $\u{k} \in \cK$. Let $f_d$ denote the homogeneous part of $f$ with highest degree $d$. We will show that this polynomial has multiplicity $l$ everywhere, so that $f$ must be the zero polynomial.
 
 Let $\u{c} \in \F_q^n$ be arbitrary and non-zero (the zero case is trivial), and take $\u{i} \in \mathbb{N}^n$ with $wt(\u{i}) < l$. As $\cK$ is conical Kakeya, there is either a parabola, hyperbola or an ellipse with direction $\u{c}$ contained in $\cK$. We split into cases, with the aim to show $(f^{\u{i}})_d(\u{c}) = 0$. 
 
 \u{\emph{Case 1 - Parabola}}
Assume there is a parabola of the form $\u{c}t^2 + \u{b}t + \u{a}$ contained in $K$. We know by the properties of $f$ that $\M(f,\u{c}t^2 + \u{b}t + \u{a}) \geq m$ for $t \in \F_q$. By Lemma \ref{weight}, we have $\M(f^{\u{i}},\u{c}t^2 + \u{b}t + \u{a}) \geq m - wt(\u{i})$. We can now use Lemma \ref{comp} to get 
$$\M(f^{\u{i}}(\u{c}t^2 + \u{b}t + \u{a}), t) \geq \M(f^{\u{i}},\u{c}t^2 + \u{b}t + \u{a}) \geq m - wt(\u{i}).$$
Note that the polynomial $f^{\u{i}}(\u{c}t^2 + \u{b}t + \u{a})$ has degree $d' \leq 2\deg(f^{\u{i}}) \leq 2(\deg(f) - wt(\u{i})) = 2(d - wt(\u{i}))$. However, $f^{\u{i}}(\u{c}t^2 + \u{b}t + \u{a})$ has multiplicity at least $m - wt(\u{i})$ everywhere in $\F_q$, so that by Lemma \ref{vanish}, as $d < lq$, $wt(\u{i}) < l$, we have 
\begin{align*}
    \sum_{t \in \F_q} \M(f^{\u{i}}(\u{c}t^2 + \u{b}t + \u{a}), t) & \geq q(m - wt(\u{i})) \\
    & > 2(d - wt(\u{i})) \\
    & \geq \deg(f^{\u{i}}(\u{c}t^2 + \u{b}t + \u{a}))
\end{align*}
so that $f^{\u{i}}(\u{c}t^2 + \u{b}t + \u{a})$ is in fact the zero polynomial.

The next observation is crucial; the coefficient of $x^{2\deg(f^{\u{i}})}$ in $f^{\u{i}}(\u{c}t^2 + \u{b}t + \u{a})$ is precisely $(f^{\u{i}})_{d'}(\u{c})$, as only this highest degree homogeneous part could reach the highest power of $x$. But then by Lemma \ref{commute}, we have
$$(f_d)^{\u{i}}(\u{c}) = (f^{\u{i}})_{d'}(\u{c}) = 0 $$
as needed.
 
 \u{\emph{Case 2 - Hyperbola}}
  Up to a relabelling of $t \rightarrow t^{-1}$, we may parametrise the hyperbola as $\u{c}t + \u{b}t^{-1} + \u{a}$. As the polynomial $f(\u{x})$ has multiplicity $m$ everywhere in $\cK$, $\M(f,\u{c}t + \u{b}t^{-1} + \u{a}) \geq m$ for $t \in \F_q^*$. We then have that $\M(f^{\u{i}},\u{c}t + \u{b}t^{-1} + \u{a}) \geq m - wt(\u{i})$ for $t \in \F_q^*$. Let $d'$ denote the degree of $f^{\u{i}}(\u{x})$. We have $d' \leq d - wt(\u{i})$, and we define the polynomial $F(t) = t^{d'}f^{\u{i}}(\u{c}t + \u{b}t^{-1} + \u{a})$. Note that $F(t)$ has multiplicity at least $m - wt(\u{i})$ for all $t \in \F_q^*$, so that by the vanishing lemma,
  $$\sum_{t \in F_q}\M(F,t) \geq (q-1)(m - wt(\u{i})) > 2(d - wt(\u{i})) \geq 2d' = \deg(F).$$
  Therefore $F(t)$ is the zero polynomial. In particular, its highest degree term is zero. The coefficient of $t^{2d'}$ in $F(t)$ is precisely $(f^{\u{i}})_{d'}(\u{c})$. By Lemma \ref{commute}, we have $(f^{\u{i}})_{d'}(\u{c}) = (f_{d})^{\u{i}}(\u{c}) = 0$, as needed.

  We now have that $\M(f_d, \u{a}) \geq l$ for all $\u{a} \in \F_q^n$. We may now use Lemma \ref{vanish} to show
$$\sum_{\u{a} \in \F_q^n}\M(f_d, \u{a}) \geq lq^n > dq^{n-1}$$
so that $f_d$ is the zero polynomial, a contradiction. We therefore must have $|\cK| \geq \left(\frac{lq -1  }{m + n}\right)^n$. As $l$ was an arbitrary large integer, we may allow $l \rightarrow \infty$, so we have
$$|\cK| \geq \lim_{l \rightarrow \infty}\left(\frac{lq -1  }{m + n}\right)^n = \lim_{l \rightarrow \infty} \left( \frac{q - 1/l }{3 + n/l} \right)^n= \left(\frac{q}{3}\right)^n $$
as needed.
 \end{proof}
 }
\section{Final remarks} \label{remarks}

\begin{itemize}
\tbd{
\item For line Kakeya sets Dvir gave a construction of size at most 
$$2^{1-n}q^n+O(q^{n-1}),$$ 
see \cite[Theorem 7]{sasu}. This construction can easily be adjusted to conical Kakeya sets. However, we lose a factor~$2$. We explain this for parabolae. For hyperbolae and ellipses one can deal analogously.
Since otherwise our result is trivial we assume $n\ge 3$. For any direction $\underline{c}=(c_1,\ldots,c_n)\ne \underline{0}$ we take $\underline{b}=(b_1,\ldots,b_n)$ with $b_1=1$ and $b_i=0$ for $i=2,\ldots,n$ if $c_n\ne 0$ and $(b_1,\ldots,b_{n-1},0)$ any vector which is linearly independent to $\underline{c}$ if $c_n=0$.
We also take $\underline{a}=(a_1,\ldots,a_n)$ with $a_n=0$. Then for $c_n = 0$ the parabola 
$\underline{a}+t\underline{b}+t^2\underline{c}$ lies in $\F_q^{n-1}\times\{0\}$ which contains $q^{n-1}$ points.
For $c_n\ne 0$ choose $b_1=1$ and $b_i=0$ for $i=2,\ldots,n$ and note that $\underline{b}$ and $\underline{c}$ are linearly independent. Choosing $a_i=c_i^2(2c_n)^{-2}$ for $i=2,\ldots,n-1$ we see that 
$\underline{a}+t\underline{b}+t^2\underline{c}$ is of the form $(\alpha_1,\ldots,\alpha_n)$
with $\alpha_i+\alpha_n^2=a_i+t^2c_i+t^4c_n^2=(t^2c_n+c_i(2c_n)^{-1})^2$ for $i=2,\ldots,n-1$ by the choice of $a_i$. Hence, the parabola lies in the set $\{(\alpha_1,\ldots,\alpha_n): \alpha_i+\alpha_n^2 \mbox{ is a square for }i=2,\ldots,n-1 \mbox{ and } \alpha_n\ne 0\}$. We have $q-1$ choices for $\alpha_n$, $q$ for $\alpha_1$ and $(q+1)/2$ for each $\alpha_i$ with $i=2,\ldots,n-1$. Hence, the size of our conical Kakeya set is at most
$$q\left(\frac{q+1}{2}\right)^{n-2}(q-1)+q^{n-1}=2^{2-n}q^n+O(q^{n-1}).$$
}

\tbd{\item In \cite[Definition 6]{bjkawi}, the authors introduced {\em Kakeya sets of degree $r$} which coincide 
with line Kakeya sets if $r=1$. For $r=2$ this definition differs from our definition of 
parabolic Kakeya sets by the condition that $\u{b}$ and $\u{c}$ are allowed to be linearly dependent.
If $q\equiv 1 \bmod (r+1)$, in Lemma~7 they also give constructions of size at most $\left(\frac{q-1}{r+1}+1\right)^{n+1}$.
For $r=2$ the construction is
$$\cK=\left\{\left(\left(\frac{c_i}{3}+t\right)^3-t^3\right)_{i=1}^n : c_1,\ldots,c_n, t\in \F_q\right\},\quad q\equiv 1\bmod 3.$$
However, to satisfy the linear independence condition we have to add lines for the directions
for which $(c_1,\ldots,c_n)$ and $(c_1^2,\ldots,c_n^2)$ are linearly dependent, that is, $(c_1,\ldots,c_n)\in \{0,c\}^n$ for some $c\in \F_q^*$ and we have to add $O(q^2)$ further vectors, that is, we have the upper bound
$$\left(\frac{q+2}{3}\right)^{n+1}+O(q^2).$$
It is not difficult to extend Theorem~\ref{parabolicKakeya} to such Kakeya sets  of degree $r$ (with a linear independence condition) giving the lower bound $\left(\frac{q}{r+1}\right)^n$. 
(Without the linear independence condition we can get only a weaker lower bound 
since the polynomial curves may contain only $\lceil q/r\rceil$ points.) }

\item  For line Nikodym sets a lower bound $(1-o(1))q^n$ is given in \cite{gukosu} where the implied constant is independent of $n$ but depends on the characteristic of $\F_q$.

\item Improved lower bounds on (line) Kakeya and Nikodym sets in $\F_q^3$ are given in \cite{lu}.
\tbd{In particular it is shown that a construction for Nikodym sets in $\F_q^3$ of size $\frac{q^3}{4}+O(q^2)$
cannot exist and Nikodym sets behave differently than Kakeya sets where we have such a construction, see our first remark.}

\item Modular conics, in particular hyperbolae, are well-studied objects. For a survey on modular hyperbolae see \cite{sh}.

\item The proofs of the lower bounds for the size of finite field Kakeya and Nikodym sets were inspired by ideas from coding theory, see for example \cite[Chapter 4]{gu} and \cite{ye}, more precisely from decoding Reed-Muller codes. The crucial idea is that a single missing value of a polynomial (of sufficiently small degree) on a line can be recovered. Similarly one can design decoding algorithms using non-degenerate conics instead of lines, see \cite[Lemma~2.6]{ye} for parabolae.

\item The following \tbd{example shows for ellipses we can neither take $\u{b}$ nor $\u{c}$ as a direction
to define elliptic Kakeya sets and prove a lower bound of order of magnitude $q^n$.
We take} $n=2$, $q\equiv 3\bmod 4$ and $q\ge 19$. 
Note that $q\equiv 3\bmod 4$ if and only if $-1$ is a non-square in $\F_q^*$, that is, for any non-square $g$ in $\F_q^*$ the element
$-g$ is a square in $\F_q^*$ and let $r\in \F_q^*$ be a square-root of $-g$, $r^2=-g$. 
Moreover, verify that $\F_q^2\setminus\{\underline{0}\}=\{(c_1r^{-1},c_2): (c_1,c_2)\in \F_q^2\setminus\{\underline{0}\} \ \}.$
Then set
$$\cK=\{x(-c_2r^{-1},c_1)+y(c_1r^{-1},c_2) : y^2=-x^2+(c_1^2+c_2^2)^{-1}, (c_1,c_2) \in \F_q^2 \setminus \{\underline{0}\} \}$$
which defines only one ellipse $\cE=\{(x,y) : y^2=gx^2+1\}$ with $q+1<\left(\frac{q-1}{4}\right)^2$ points by \cite[Lemma~6.24]{lini} and since $q\ge 19$.
Note that $c_1^2+c_2^2=0$ with $(c_1,c_2)\ne (0,0)$ is not possible since $-1$ is a non-square in $\F_q$ for $q\equiv 3\bmod 4$. 
This example explains why we did not include the case of ellipses into the definition of conical Kakeya sets.
\end{itemize}

\section*{Acknowledgment}
The authors are supported by the Austrian Science Fund FWF Project P~30405-N32.\\
\tbd{We would like to thank the anonymous referees for their valuable comments, in particular for pointing to~\cite{elobta} and \cite{bjkawi}.}

\end{document}